\documentclass[11pt]{article}

\usepackage{amsmath, amssymb, amsthm, amscd, geometry, amsfonts, graphicx, fancyhdr, xcolor}
\usepackage{hyperref}

\newtheorem{thm}{Theorem}[section]
\newtheorem{lemma}[thm]{Lemma}

\newtheorem{remark}[thm]{Remark}

\begin{document}

\title{Reconstruction of the Initial Condition for a Non-Homogeneous Heat Equation from Finite Measurements}


\author{
  R. Dollente, S. Frerichs, H. Nam, \\
  J. Shen, G. Walters, and Y. You \\[1em]
  \small Department of Mathematics \\
  \small Indiana University East \\
  \small Richmond, IN 47374, USA \\[0.5em]
  \small \texttt{\{radollen, sfrerich, heenam, jsh7, grmcwalt, youy\}@iu.edu}
}

\date{}

\maketitle

\begin{abstract}
We address the inverse problem of reconstructing the initial temperature distribution in a one-dimensional non-homogeneous heat equation with Dirichlet boundary conditions from a finite number of pointwise-in-time measurements at a single spatial location. We propose an explicit approximation scheme that incorporates the source term and utilizes a refined sequence of measurement times. Unlike previous approaches that often rely on exponentially growing observation times, our chosen time sequence allows for efficient recovery within a fixed time horizon. We establish algebraic convergence rates under suitable regularity assumptions on the initial data and the forcing, and we validate the theoretical sharpness with numerical experiments for varying numbers of measurements.
\end{abstract}


\medskip


\section{Introduction}

The reconstruction of initial data for parabolic equations from partial observations is a classical inverse problem with applications in heat conduction, diffusion processes, and control of distributed parameter systems; see, e.g., \cite{IsakovBook}. In the homogeneous one-dimensional heat equation, an influential contribution of DeVore and Zuazua \cite{DeVoreZuazua2014} provides a constructive procedure to recover the initial temperature from a finite sequence of measurements at a single spatial point, together with sharp convergence rates in suitable Sobolev scales.

Recently, there has been renewed interest in dynamical sampling and related strategies for recovering initial conditions or system parameters from later-time measurements. Aceska, Kim, and Pallage \cite{AceskaKimPallage2025} developed a general system-approximation framework based on samples collected at later times along a generalized temporal grid, allowing for non-uniform sampling patterns and time-variant systems. In a closely related work \cite{AceskaKimPallage2024}, they studied the specific problem of recovering an initial condition from later time samples, establishing identifiability conditions and convergence results for reconstruction algorithms that exploit the spectral properties of the underlying evolution operator.

On the numerical analysis side, Souleymane and Ammari \cite{SouleymaneAmmari2022} proposed a wavelet-based Galerkin method for identifying initial conditions in boundary value problems. Their approach expands the unknown initial data in a wavelet basis and solves the resulting discrete system, yielding accurate reconstructions and rigorous error estimates. While their methodology differs from the spectral recursion considered here, it highlights the effectiveness of multiresolution techniques for initial-condition identification.

In many practical applications, however, the underlying diffusion is driven by a non-trivial source term, meaning the dynamics are governed by a non-homogeneous heat equation. In this setting, the measurements contain contributions from both the unknown initial condition and the known forcing. Separating these contributions is essential to obtain stable reconstruction schemes. The goal of this paper is to extend the discrete sampling framework of \cite{DeVoreZuazua2014} to a non-homogeneous one-dimensional heat equation with Dirichlet boundary conditions.

A key contribution of our work is the use of a refined measurement time sequence. While the original DeVore--Zuazua strategy typically employs an exponentially growing sequence of times which may require measurements over a large time interval, our proposed sequence stays within a controlled horizon while providing sufficient data density to stably invert the heat operator. This choice improves the conditioning of the recursive reconstruction algorithm and allows for explicit error control even when the source term $F(x,t)$ is present.

We derive an explicit recursive algorithm which approximates the Fourier coefficients of the initial data from finitely many measurements and incorporates the contribution of the forcing through a suitable truncation parameter depending on the regularity of the source. Under assumptions on the Sobolev regularity of the initial data and on the spatial regularity of the forcing, we prove algebraic convergence of the reconstruction error.

\medskip

The remainder of this paper is structured as follows.
Section~\ref{sec:problem} introduces the non-homogeneous model and the functional setting. Section~\ref{sec:scheme} presents the reconstruction scheme and the main intermediate estimates. Section~\ref{sec:main-results} states and proves the convergence theorem. Section~\ref{sec:numerics} reports numerical experiments illustrating the performance of the method. Concluding remarks are discussed in Section~\ref{sec:conclusion}.

\section{Problem setting}
\label{sec:problem}

We consider the non-homogeneous heat equation with Dirichlet boundary conditions:
$$
\begin{cases}
u_t = u_{xx} + F(x,t), & 0 < x < \pi,\ t > 0,\\
u(0,t) = u(\pi,t) = 0, & t > 0,\\
u(x,0) = f(x), & 0 < x < \pi,
\end{cases}
$$
where $f$ is the unknown initial temperature and $F$ is a given source term. Throughout the paper, we assume that $F$ is sufficiently regular in space and time so that the series representations below are well defined.

Given a finite set of measurements $u(x_0, t_1), \dots, u(x_0,t_n)$ at a fixed spatial location $x_0 \in (0,\pi)$, our goal is to approximate the initial condition $f$.

We recall that the solution admits the standard sine-series representation
\[
u(x,t)
= \sum_{j=1}^{\infty} e^{-j^2 t} \hat{f}_j \sin(jx)
+ \sum_{j=1}^{\infty} e^{-j^2 t}
\left( \int_0^t e^{j^2 s} \hat{F}_j(s) \, ds \right) \sin(jx),
\]
where $\hat{f}_j$ and $\hat{F}_j(s)$ denote the Fourier sine coefficients of $f$ and $F(\cdot,s)$, respectively:
\[
\hat{f}_j = \frac{2}{\pi}\int_0^{\pi} f(x) \sin(jx) \,dx,\quad
\hat{F}_j(s) = \frac{2}{\pi}\int_0^{\pi} F(x,s) \sin(jx) \,dx.
\]
Evaluating at $x=x_0$ and $t=t_k$ yields
\begin{equation}\label{eq:ux0tk}
u(x_0, t_k)
= \sum_{j=1}^{\infty} e^{-j^2 t_k} \hat{f}_j \sin(j x_0)
+ \sum_{j=1}^{\infty} e^{-j^2 t_k}
\left( \int_0^{t_k} e^{j^2 s} \hat{F}_j(s) \, ds \right) \sin(j x_0).
\end{equation}
Noting that $u(x,0) = f(x) = \sum_{j=1}^{\infty} \hat{f}_j \sin(jx)$, we construct a truncated series approximation $\bar{f}_n$ using the first $\lceil n/2 \rceil$ modes, with coefficients $\bar{\hat{f}}_j$ determined recursively from the measurements:
\[
\bar{f}_n(x) = \sum_{j=1}^{\lceil n/2 \rceil} \bar{\hat{f}}_j \sin(jx).
\]

\section{Approximation scheme for the initial condition}
\label{sec:scheme}

We define the function class $\mathcal{F}_r$ for $r \ge 2$ as
\[
\mathcal{F}_r
:= \left\{ f \in L^2(0,\pi)\;\middle|\;
\sum_{j=1}^\infty j^{2r} |\hat{f}_j|^2 \le 1
\right\}.
\]
This describes a unit ball in a Sobolev-type space.

\begin{lemma}\label{lem:1}
If $f \in \mathcal{F}_r$, then for $k = 1,2,\ldots$,
\[
\sum_{j=k+1}^{\infty} \left|\hat{f}_j \sin(j x_0)\right|
\le \frac{k^{-r+1}}{r-1}.
\]
\end{lemma}

\begin{proof}
If $f \in \mathcal{F}_r$, then $|\hat{f}_j \sin(j x_0)| \le |\hat{f}_j| \le j^{-r}$ for $j \ge 1$. The result follows by estimating the sum with the integral $\int_k^\infty x^{-r}\,dx$.
\end{proof}

Isolating the $k$-th term in the spectral expansion \eqref{eq:ux0tk} at time $t_k$, we have 
\begin{align*}
\hat{f}_1 \sin(x_0)
&= e^{t_1} \Bigg[
u(x_0, t_1)
- \sum_{j=2}^\infty e^{-j^2 t_1} \hat{f}_j \sin(j x_0)
- \sum_{j=1}^{\infty} e^{-j^2 t_1}
\left( \int_0^{t_1} e^{j^2 s} \hat{F}_j(s) ds \right) \sin(j x_0)\Bigg],\\
\hat{f}_k \sin(k x_0)
&= e^{k^2 t_k} \Bigg[
u(x_0, t_k)
- \sum_{j=1}^{k-1} e^{-j^2 t_k} \hat{f}_j \sin(j x_0)
- \sum_{j=k+1}^\infty e^{-j^2 t_k} \hat{f}_j \sin(j x_0) \\
&\hspace{3.5cm}
- \sum_{j=1}^{\infty} e^{-j^2 t_k}
\left( \int_0^{t_k} e^{j^2 s} \hat{F}_j(s) ds \right) \sin(j x_0)\Bigg],
\quad k\ge2.
\end{align*}
to recursively define $\bar{\hat{f}}_k$, truncated approximations for $\hat{f}_k$. In particular, for $k=1$,
\begin{equation}\label{eq:fhb1}
\bar{\hat{f}}_1
:= \frac{e^{t_1}}{\sin(x_0)}
\left[
u(x_0, t_1)
- \sum_{j=1}^{n_1} e^{-j^2 t_1}
\left( \int_0^{t_1} e^{j^2 s} \hat{F}_j(s)ds \right) \sin(jx_0)
\right],
\end{equation}
and, for $k \ge 2$,
\begin{equation}\label{eq:fhbk}
\bar{\hat{f}}_k
:= \frac{e^{k^2 t_k}}{\sin(kx_0)}
\left[
u(x_0, t_k)
- \sum_{j=1}^{k-1} e^{-j^2 t_k} \bar{\hat{f}}_j \sin(jx_0)
- \sum_{j=1}^{n_k} e^{-j^2 t_k}
\left( \int_0^{t_k} e^{j^2 s} \hat{F}_j(s)ds \right) \sin(jx_0)
\right],
\end{equation}
where the truncation index $n_k$ is defined by
\begin{equation}\label{def:n_k}
n_k := C e^{(k+1)^2 t_k/2},\qquad
C := \frac{2}{\pi}\max_{0\le s\le T}
\left\|\frac{\partial}{\partial x}F(\cdot,s)\right\|_{L^1(0,\pi)}.
\end{equation}
The parameter $n_k$ ensures that the tail of the source term is negligible relative to the reconstruction error.

\section{Main results}
\label{sec:main-results}

\begin{lemma}\label{lem:2}
Let $T>0$ and define the decreasing sequence of measurement times
\begin{equation}\label{def:tj}
t_j = \frac{(2j -1)!}{8^{j-1} j! (j -1)!} T,\qquad j = 1, 2, \dots, n.
\end{equation}
If $f \in \mathcal{F}_r$ with $r \ge 2$ and $F \in C^1((0,\pi]\times[0,T])$ satisfies homogeneous boundary conditions, then
\begin{equation}\label{ineq:fmfs}
\left|(\hat{f}_j - \bar{\hat{f}}_j) \sin(jx_0)\right|
\le 2^j e^{-(2j+1)t_j},\qquad j=1,2,\dots,n.
\end{equation}
\end{lemma}

\begin{remark}[Choice of time sequence]
The specific choice of $t_j$ in \eqref{def:tj} differs from the exponentially growing sequences (e.g., $t_j \sim \rho^j$) often used in previous works such as \cite{DeVoreZuazua2014}. Our sequence remains bounded within a practical horizon $T$ while providing the specific decay rate required to counteract the growth of the condition number in the recursive inversion. This improves the quality of the reconstruction by avoiding the need for measurements at arbitrarily large times.
\end{remark}

\begin{proof}
The proof follows the induction strategy developed in \cite{DeVoreZuazua2014} for the homogeneous case, with improved choice of time sequence for measurements and additional estimates for the contribution of the source controlled via the choice of $n_k$ in \eqref{def:n_k}. For completeness, we reproduce the main steps, adapting the argument to the non-homogeneous setting.

We start from verifying the desired inequality \ref{ineq:fmfs} for $j=1$. In fact, LEMMA \ref{lem:1} applies to get
\begin{eqnarray*}
\left|(\hat{f}_1 - \bar{\hat{f}}_1)\sin(x_0)\right| &=& 
\left| e^{t_1}\sum_{j=2}^\infty e^{-j^2 t_1} \hat{f}_j \sin(j x_0)
 +  e^{t_1} \sum_{j=n_1+1}^{\infty} e^{-j^2 t_1} \left( \int_0^{t_1} e^{j^2 s} \hat{F}_j(s) ds \right) \sin(j x_0)
\right| \\
&\le& \sum_{j=2}^\infty  |\hat{f}_j\sin(jx_0)| e^{(1-j^2) t_1} 
+ \sum_{j=n_1+1}^{\infty} e^{(1-j^2) t_1} \left( \int_0^{t_1} e^{j^2 s} |\hat{F}_j(s)| ds \right) \\
&\le&  \frac{1}{r-1}e^{-3 t_1} + 
\sum_{j=n_1+1}^{\infty} e^{(1-j^2) t_1} 
\max_{0\le s\le t_1}|\hat{F}_j(s)| \left( \int_0^{t_1} e^{j^2 s}  ds\right)\\
&\le&  e^{-3 t_1} + 
\sum_{j=n_1+1}^{\infty} \frac{e^{t_1}}{j^2} 
\max_{0\le s\le t_1}|\hat{F}_j(s)| \\
&\le& 2 e^{-3 t_1}
\end{eqnarray*}
where, in the last inequality, we used the choice of $n_1$ from (\ref{def:n_k}) based on the decay estimate of $|\hat{F}_j(s)|$, e.g. taking the integration by parts,
$$|\hat{F}_j(s)|=\frac{2}{\pi} \int_0^{\pi}
\left|\frac{\partial}{\partial x}F(x,s) \frac{\cos(jx)}{j}\right| dx
\le\frac{2}{j\pi}\left\|\frac{\partial}{\partial x}F(x,s)\right\|_{L^1(0,\pi)}
\le \frac{C_1}{j}.
$$
For $2\le k \le n$, suppose that the inequality (\ref{ineq:fmfs}) holds for $j=1, 2, ..., k-1$. Then
\begin{eqnarray*}
\left|(\hat{f}_k - \bar{\hat{f}}_k)\sin(k x_0)\right| 
 &=& \left| - \sum_{j=1}^{k-1} e^{(k^2-j^2)t_k} (\hat{f}_j-\bar{\hat{f}}_j) \sin(j x_0)  
 - \sum_{j=k+1}^\infty e^{(k^2-j^2)t_k} \hat{f}_j \sin(j x_0)\right. \\ 
 &&\qquad \left. +  \sum_{j=n_k+1}^{\infty} e^{(k^2-j^2) t_k} \left( \int_0^{t_k} e^{j^2 s} \hat{F}_j(s) ds \right) \sin(j x_0) \right| \\
&\le& 
  \sum_{j=1}^{k-1} e^{(k^2-j^2)t_k} \left|(\hat{f}_j-\bar{\hat{f}}_j) \sin(j x_0)\right|  
 + \sum_{j=k+1}^\infty e^{(k^2-j^2)t_k} |\hat{f}_j \sin(j x_0)|\\
 &&\qquad  +  \sum_{j=n_k+1}^{\infty} e^{(k^2-j^2) t_k} \left( \int_0^{t_k} e^{j^2 s} |\hat{F}_j(s)| ds \right) \\ 
&\equiv& \mbox{(I)} + \mbox{(II)} + \mbox{(III)}
 \end{eqnarray*}
Note that the choice of the sequence $t_j$ leads to
\[
((k+1)^2-j^2)t_k-(2j+1)t_j\le0, \quad j=1,2, ..., k,
\]
which can be shown by backward induction method.
Therefore (I) can be estimated as
\begin{eqnarray*}
\sum_{j=1}^{k-1} e^{(k^2-j^2)t_k} \left|(\hat{f}_j-\bar{\hat{f}}_j) \sin(j x_0)\right|
&\le&   
  \sum_{j=1}^{k-1}2^j e^{(k^2-j^2)t_k-(2j+1)t_j} \\
&=&   
e^{-(2k+1)t_k} \sum_{j=1}^{k-1}2^j e^{((k+1)^2-j^2)t_k-(2j+1)t_j} \\
&\le&
e^{-(2k+1)t_k} \sum_{j=1}^{k-1}2^j 
= e^{-(2k+1)t_k} (2^k-2).
\end{eqnarray*}
For (II), LEMMA \ref{lem:1} applies to get
\[
\sum_{j=k+1}^\infty e^{(k^2-j^2)t_k} |\hat{f}_j \sin(j x_0)|
\le \sum_{j=k+1}^\infty e^{-(2k+1)t_k} |\hat{f}_j \sin(j x_0)|
\le \frac{k^{-r+1}}{r-1}e^{-(2k+1)t_k} 
\le e^{-(2k+1)t_k}.
\]
(III) can be estimated using the similar argument for the case $j=1$:
$$\sum_{j=n_k+1}^{\infty} e^{(k^2-j^2) t_k} \left( \int_0^{t_k} e^{j^2 s} |\hat{F}_j(s)| ds \right)\le e^{-(2k+1)t_k}.$$
Combining the cases (I) through (III), we obtained the desired inequality (\ref{ineq:fmfs}) for $j=k$ which completes the induction method.

\end{proof}

For the recovery process, we require the spatial point $x_0$ to satisfy a Diophantine condition. We select $x_0 \in (0,\pi)$ such that
\begin{equation}\label{def:x0}
\mathrm{dist}(kx_0, \{0,\pi,2\pi,\dots\}) \ge \frac{c_0 \pi}{k}, \quad k=1,2,\dots,
\end{equation}
where $c_0 > 0$. The existence of such points is guaranteed by results in Diophantine approximation; e.g. Bugeaud \cite[p.~245]{Bugeaud2012}.

\begin{thm}\label{thm:main}
Let $f \in \mathcal{F}_r$ ($r\ge2$) and $F\in C^1$ as above. Suppose $x_0$ satisfies \eqref{def:x0}, and measurements are taken at times $t_k$ given by \eqref{def:tj}. Then for sufficiently large $n$, the approximation $\bar{f}_n(x) = \sum_{j=1}^{\lceil n/2\rceil} \bar{\hat{f}}_j \sin(jx)$ satisfies
\[
\|f - \bar{f}_n\|_{L^2(0,\pi)}
\le C(x_0,r)\,n^{-r},
\]
for a constant $C(x_0,r)>0$ independent of $n$.
\end{thm}

\begin{proof}
Using \eqref{def:x0}, we have
$$ 
|\sin(kx_0)| \ge \frac{2c_0}{k},\qquad k=1,2,\dots.
$$
Parseval's identity and Lemma~\ref{lem:2} imply
\begin{eqnarray*}
\| f - \bar{f}_n \|_{L^2}^2 
  &=&\frac{\pi}{2}\left(\sum_{k=1}^{\lceil n/2 \rceil} |\hat{f}_k - \bar{\hat{f}}_k|^2 
       + \sum_{k=\lceil n/2 \rceil+1}^{\infty} |\hat{f}_k|^2\right) \\ 
  &\le&\frac{\pi}{2}\left(
  \sum_{k=1}^{\lceil n/2 \rceil} \frac{1}{|\sin(kx_0)|^2} 2^{2k} e^{-2(2k+1)t_k} 
  + \sum_{k=\lceil n/2 \rceil+1}^{\infty}\frac{1}{(\lceil n/2 \rceil+1)^{2r}} k^{2r}|\hat{f}_k|^2\right) \\ 
  &\le&\frac{\pi}{2}\left(
  \sum_{k=1}^{\lceil n/2 \rceil} \frac{k^2}{4c_0^2} 2^{2k} e^{-2(2k+1)t_k} 
  + \frac{2^{2r}}{n^{2r}}\right) \\ 
  &\le&\frac{\pi}{2}\left(\frac{e^{-6t_{\lceil n/2 \rceil}}}{4c_0^2}
  \sum_{k=1}^{\lceil n/2 \rceil} k^2 4^{k}
  + \frac{2^{2r}}{n^{2r}}\right) \\ 
 &\le&\frac{\pi}{2}\left(\frac{e^{-6t_{\lceil n/2 \rceil}+2r\ln n}}{4c_0^2}
  \sum_{k=1}^{\lceil n/2 \rceil}n^2 4^k
  + 4^{r}\right) {n^{-2r}}\\ 
 &\le&\frac{\pi}{2}\left(\frac{1}{3{c_0^2}} e^{-6t_{\lceil n/2 \rceil}+2(r+1)\ln n+ 
       2 \lceil n/2 \rceil \ln 2 }
   + 4^{r}\right){n^{-2r}} \\ 
  &\le& C(x_0,r) n^{-2r}.
\end{eqnarray*}
\end{proof}

\section{Numerical experiments}
\label{sec:numerics}

In this section, we illustrate the reconstruction scheme. We consider the target initial condition
\[
f(x) = \frac{1}{8}\sin(2x) + \frac{1}{18}\sin(3x)
\]
and the non-homogeneous source term
\[
F(x,t) = e^{-t} \sin x.
\]
Using analytically generated measurements, we examine the convergence of the reconstruction as the number of measurements $n$ increases. See Figure \ref{fig:reconstruction} which displays the true initial condition $f(x)$ along with the reconstructed approximations $\bar{f}_n(x)$ for $n \in \{2, 4, 10\}$.

\begin{figure}[htbp]
\centering
\includegraphics[width=0.8\textwidth]{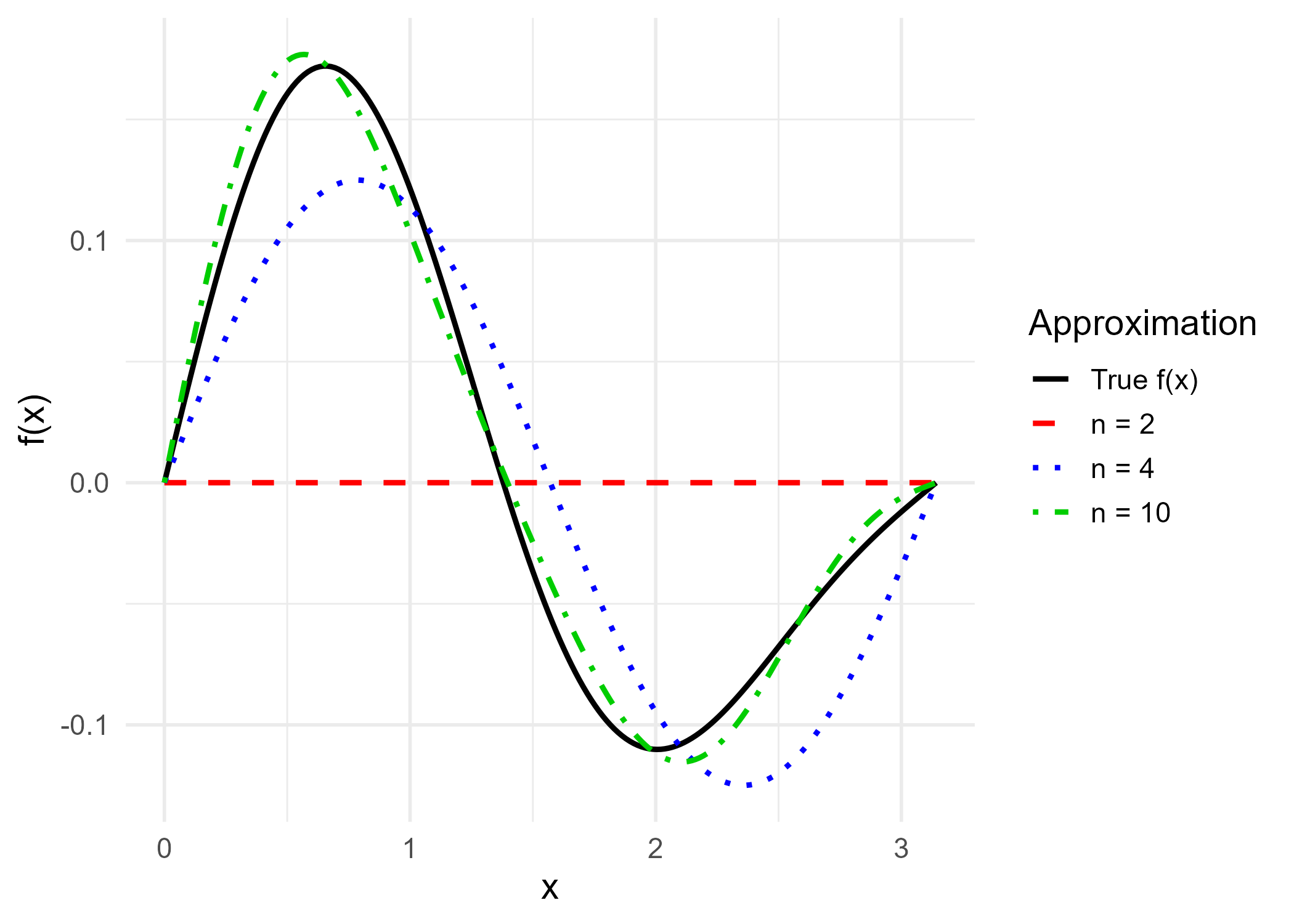}
\caption{Reconstruction of the initial condition for $n=2$, $n=4$, and $n=10$.}
\label{fig:reconstruction}
\end{figure}

For $n=2$, the reconstruction captures the low-frequency trend but misses details. At $n=4$, the shape is significantly improved. By $n=10$, the reconstructed curve shows a good approximation of the true initial condition, confirming the theoretical convergence rates.

\section{Conclusions}
\label{sec:conclusion}

We have presented a constructive algorithm for recovering the initial temperature of a non-homogeneous heat equation from discrete pointwise measurements. By carefully accounting for the source term's contribution and utilizing a refined measurement time sequence, we extended the results of DeVore and Zuazua to the non-homogeneous setting. The proposed time sequence offers practical advantages by keeping measurement times within a fixed horizon while maintaining stability. Theoretical analysis proves algebraic convergence rates, and numerical tests validate that the method yields accurate results with a feasible number of measurements.

\section*{Acknowledgments}
This work was partially supported by the National Science Foundation under Grant No. 2015553, the Pi Mu Epsilon Honors Society, the IU East Undergraduate Student Research Expense Fund, and the IU East Faculty Research \& Creative Activity Support Fund.


\end{document}